\documentclass[a4paper,12pt]{amsart}
\usepackage{amsmath,amsthm,amssymb}
\usepackage{hyperref}

\allowdisplaybreaks[1]

\newcommand{\ti}{\widetilde}

\newcommand{\ity}{\infty}

\newcommand{\C}{\mathbb{C}}

\newcommand{\N}{\mathbb{N}}

\newcommand{\be}{\beta}
\newcommand{\al}{\alpha}
\newcommand{\ze}{\zeta}

\numberwithin{equation}{section}

\newtheorem{theorem}{Theorem}[section]
\newtheorem{lemma}[theorem]{Lemma}
\newtheorem{corollary}[theorem]{Corollary}
\theoremstyle{remark}

\thanks {The research work of the   author is supported by research fellowship from Council of Scientific and Industrial Research (CSIR), New Delhi.}

\begin{document}
\title[dynamics and semiconjugation]{On dynamics of semiconjugated entire functions}
\author[D. Kumar]{Dinesh Kumar}
\address{Department of Mathematics, University of Delhi,
Delhi--110 007, India}

\email{dinukumar680@gmail.com }
%
%

\begin{abstract}
Let $g$ and $h$ be transcendental entire functions and let $f$ be a continuous map of the complex plane into itself with $f\circ g=h\circ f.$ Then $g$ and $h$ are said to be semiconjugated by $f$ and $f$ is called a semiconjugacy.
We consider the dynamical properties of semiconjugated transcendental entire functions $g$ and $h$ and provide several conditions under which the semiconjugacy $f$ carries Fatou set of one entire function into the Fatou set of other entire function appearing in the semiconjugation. We have also shown that under certain condition on the growth  of  entire functions appearing in the semiconjugation, the set of  asymptotic values of the derivative of composition of the entire functions is bounded.
\end{abstract}


\maketitle

\section{Introduction}\label{sec1}

 Let $f$ be a transcendental entire function. For $n\in\N$ let $f^n$ denote the nth iterate of $f.$ The set $F(f)=\{z\in\C : \{f^n\}_{n\in\N}\,\text{ is normal in some neighborhood of}\, z\}$ is called the Fatou set of $f$ or the set of normality of $f$ and its complement $J(f)$ is the Julia set of $f$.  The Fatou set is open and completely invariant: $z\in F(f)$ if and only if $f(z)\in F(f)$ and consequently $J(f)$ is completely invariant. The Julia set of a transcendental entire function is non empty, closed perfect set and unbounded. All these results and more can be found in Bergweiler \cite{berg1}. If $U$ is a component of Fatou set $F(f),$ then $f(U)$ lies in some component $V$ of $F(f)$ and  $V\setminus f(U)$ is a set which contains atmost one point, \cite{berg4}. This result was also proved in \cite{MH} independently.
A component $U$ of  Fatou set of $f$ is called a wandering domain if $U_k\cap U_l=\emptyset$ for $k\neq l,$ where $U_k$ denotes the component of $F(f)$ containing $f^k(U).$ Sullivan \cite{Sullivan} proved that  Fatou set of any rational function has no wandering domains. Transcendental entire functions may however have wandering domains, but certain classes of functions which do not have wandering domains are known, \cite{baker3, baker7, berg1, berg2, el2, keen, stallard}. 
 A component $U\subset F(f)$ is called  preperiodic if $f^k(U_l)=U_l$ for some $k,l\geq 0.$ If $f^k(U)=U,$ for some $k\in\N,$ then $U$ is called a periodic component of $F(f)$ and there are then four possibilities \cite{baker4}, as the possibility of a Herman ring is ruled out for a transcendental entire function, \cite[p.\ 65]{Hua}. In particular, $U$ is called a Baker domain if $f^{nk}(z)\to\ity$ as $n\to\ity$ for $z\in U$ and, moreover, $U$ is simply connected \cite[Theorem 3.1]{baker3}.
 For more details on the subject we refer the reader to \cite{ beardon,berg1,Hua}.

Now let $g$ and $h$ be entire functions and let $f:\C\to\C$ be a continuous function such that
\[f\circ g=h\circ f.\]
Then we say that $g$ and $h$ are semiconjugated (by $f$) and $f$ is called a semiconjugacy. In particular if $g=h,$ that is, $f\circ g=g\circ f$ and $f$ is entire, then $f$ and $g$ are said to be permutable. There have been several results on the dynamics of permutable entire functions. For instance, Fatou  \cite{beardon}, proved that if $f$ and $g$ are two rational functions which are permutable, then $F(f)=F(g)$. This was the most famous result  that motivated the dynamics of composition of complex functions. Similar results for transcendental entire functions is still not known, though it holds in some very special cases,  \cite[Lemma 4.5]{baker3}.
If $f$ and $g$ are transcendental entire functions, then so is $f\circ g$ and $g\circ f$ and the dynamics of one composite entire function helps in the study of the dynamics of the other and vice-versa. 
If $f$ and $g$ are transcendental entire functions, the dynamics of $f\circ g$ and $g\circ f$ are very similar. In  \cite{berg5, Poon}, it was shown $f\circ g$ has  wandering domains if and only if $g\circ f$ has  wandering domains. In \cite{dinesh1}     the authors have constructed several  examples where the dynamical behavior of $f$ and $g$ vary greatly from the dynamical behavior of $f\circ g$ and $g\circ f.$ Using approximation theory of entire functions, the authors have shown the existence of entire functions $f$ and $g$ having infinite number of domains satisfying various properties and relating it to their composition. They explored and enlarged all the maximum possible ways of the solution in comparison to the past result worked out.

Recall that $w\in\C$ is a critical value of a transcendental entire function $f$ if there exist some $w_0\in\C$ with $f(w_0)=w$ and $f'(w_0)=0.$ Here $w_0$ is called a critical point of $f.$ The image of a critical point of $f$ is  critical value of $f.$ Also recall that $\zeta\in\C$ is an asymptotic value of a transcendental entire function $f$ if there exist a curve $\Gamma$ tending to infinity such that $f(z)\to \zeta$ as $z\to\ity$ along $\Gamma.$ The set of all asymptotic values of $f$ will be denoted by   $AV(f)$. In \cite{dinesh2} the authors have considered the dynamical properties of transcendental entire functions $f, g$ and their compositions. They have given  several conditions under which Fatou set of  $f$ coincide with that of $f\circ g.$  They have  also proved some result giving relationship between singular values of transcendental entire functions and their compositions.

Recall the Eremenko-Lyubich class
 \[\mathcal{B}=\{f:\C\to\C\,\,\text{transcendental entire}: \text{Sing}{(f^{-1})}\,\text{is bounded}\},\]
(where Sing$f^{-1}$ is the set of critical values and asymptotic values of $f$ and their finite limit points). Each $f\in\mathcal{B}$ is said to be of bounded type. A transcendental entire function $f$ is of finite type if Sing$f^{-1}$ is a finite set. Furthermore if the transcendental entire functions $f$ and $g$ are of bounded type then so is $f\circ g$ as Sing $((f\circ g)^{-1})\subset$ Sing $f^{-1}\cup f(\text{Sing}(g^{-1})),$ \cite{berg5}. Singularities of a transcendental map plays an important role in its dynamics. For any transcendental entire function Sing$f^{-1}\neq\emptyset,$ \cite[p.\ 66]{Hua}. It is well known \cite{el2, keen}, if $f$ is of finite type then it has no wandering domains. Recently Bishop \cite{bishop} has constructed an example of a function of bounded type having a wandering domain. Let $E(f)=\cup_{n\geq 0}f^{n}(\text{Sing}f^{-1})$ and $E'(f)$ be the derived set of $E(f),$ that is, the set of finite limit points of $E(f).$ It is well known  \cite{berg2}, if $U\subset F(f)$ is a wandering domain, then all limit functions of $\{f^n\vert_{U}\}$ are constant and are contained in $(E'(f)\cap J(f))\cup\{\ity\}.$ Furthermore if  $C$ denotes the class of transcendental entire functions with $J(f)\cap E'(f)=\emptyset,$ and if $f\in\mathcal{B}\cap C,$ then $f$ does not have any wandering domains, \cite[Corollary]{berg2}. 
 
The order of a transcendental entire function $f$ is defined as 
\[\rho_f=\limsup_{r\to\ity}\frac{\log \log M(r, f)}{\log r}\]
where 
\[M(r, f)=max\{|f(z)|: |z|\leq r\}.\]

Here we shall consider the dynamics of semiconjugated entire functions. Some results are motivated by work in \cite{ap1}. We provide several conditions under which the semiconjugacy  carries Fatou set of one entire function into  Fatou set of other entire function appearing in the semiconjugation. We  also show that under certain condition on the growth  of  entire functions appearing in the semiconjugation, the number of  asymptotic values of the derivative of composition of the entire functions is bounded.

\section{Theorems and their proofs}\label{sec2}

\begin{theorem}\label{sec2,thm1}
Let $g$ and $h$ be transcendental entire functions and $f$ be continuous and open such that $f\circ g=h\circ f.$ Suppose $g$ has neither wandering domains nor Baker domains. Then $f(F(g))\subset F(h).$ 
\end{theorem}
For the proof we will need the following lemmas:
\begin{lemma}\cite[Lemma 2.1]{baker3}\label{sec2,lem1} 
Let $f$ be a transcendental entire function, $z_0\in J(f)$ and let $w\in\C$ which is not a Fatou exceptional value of $f.$ Then there exist a sequence $z_{n_k}, n_k\in\N$ such that $f^{n_k}(z_{n_k})=w, z_{n_k}\to z_0,$ as $n_k\to\ity.$
\end{lemma}
\begin{lemma}\cite[Lemma 2.2]{baker3}\label{sec2,lem2}
Let $f$ be a transcendental entire function. Let $\zeta\in J(f), W$ be an open neighborhood of $\zeta$ and suppose $K$ is a compact plane set which does not contain any Fatou exceptional value of $f$. Then there exist an $m$ such that $f^n(W)\supset K$ for all $n>m.$
\end{lemma}
\begin{lemma}\cite[Lemma 2.3]{baker3}\label{sec2,lem3}
Let $f$ and $g$ be transcendental entire functions. Then
\begin{enumerate}
\item [(i)] if $\al\in F(f)$ and there is a subsequence ($f^{n_k}$), with $n_k\to\ity,$ which has a finite limit in the component of $F(f)$ which contains $\al,$ then $g(\al)\in F(f)$;
\item [(ii)] if $\ity$ is not a limit function of any subsequence of ($f^n$) in a component of $F(f)$, then $g(F(f))\subset F(f).$
\end{enumerate}
\end{lemma}
We now give the proof of Theorem 2.1
\begin{proof}
As $g$ has neither wandering domains nor Baker domains, so all limit functions of ($g^n$) on all components of $F(g)$ are finite. Let $\be\in F(g)$ and let $V$ be an open neighborhood of $\be$ such that $\overline V\subset F(g).$ Using Lemma \ref{sec2,lem3},  $\cup_{n=0}^{\ity} g^n(V)$ lie in some compact plane set $K$ on which $f$ is uniformly continuous. We can choose $V$ small enough so that $f(g^n(V))=h^n(f(V))$ has small diameter for $n$ large. Thus it is possible to  choose a point $\ze$ which is not a Fatou exceptional value of $h$ and which lies outside $f(g^n(V))$ for $n$ large. Suppose $f(\be)\in J(h).$ As $f$ is open, so $f(V)$ is an open set and $f(\be)\in f(V).$ Using Lemma \ref{sec2,lem2}, there exist   $m\in\N$ such that $h^n(f(V))\supset K$ for all $n>m.$ From Lemma \ref{sec2,lem1}, there exist a sequence $z_{n}\to f(\be), z_{n}\subset f(V)$ such that $h^n(z_n)=\ze$ as $n\to\ity.$  As $\ze\notin f(g^n(V))=h^n(f(V))$ for $n$ large, we obtain $h^n(z_n)\nsubseteqq h^n(f(V))$ for $n$ large, a contradiction and hence the result.
\end{proof}
An immediate consequence of above theorem is
\begin{corollary}\label{sec2,cor1}
Let $g$ and $h$ be transcendental entire functions and $f$ be continuous and open such that $f\circ g=h\circ f.$ If $g\in\mathcal  B\cap C,$ then  $f(F(g))\subset F(h).$ 
\end{corollary}

\begin{theorem}\label{sec2,thm2}
Let $g$ and $h$ be transcendental entire functions and $f$ be continuous and open such that $f\circ g=h\circ f.$ Suppose $f$ omits two values in the complement of closure of each disk in $\C,$ then $f(F(g))\subset F(h).$
\end{theorem}

\begin{proof}
Let $\be\in F(g)$ and suppose $f(\be)\notin F(h),$ that is, $f(\be)\in J(h).$ Select an open neighborhood $V$ of $\be$ such that $\overline V\subset F(g).$ As ($g^n$) is normal on $V,$ every sequence in ($g^n$) has  a subsequence  which converges locally uniformly on $V$ to either an analytic function say $G,$  or to $\ity.$ Let $g^{n_k}$ be a subsequence of ($g^n$) such that $g^{n_k}\to G$ on $V$ as $n_k\to\ity$. Then the argument used in the proof of Theorem \ref{sec2,thm1} can be used verbatim to arrive at the conclusion. Now suppose $g^{n_k}\to\ity$ on $V$ as $n_k\to\ity.$ Then for all $z\in V$ and any $A>0,$ there exist $m\in\N$ such that $|g^{n_k}(z)|>A$ for all $n_k>m.$ Then for all $n_k>m$ and for all $z\in V, |h^{n_k}(f(z))|=|f(g^{n_k}(z))|=|f(\zeta_{n_k})|,$ where $|\ze_{n_k}|=|g^{n_k}(z)|>A.$ Thus for all $n_k>m$ each $h^{n_k}$ omits two points on $f(V)$ by hypothesis and hence by Montel's Normality Criterion, ($h^n$) is normal on $f(V)$. This is a contradiction to $f(\be)\in J(h)$ and hence the result.
\end{proof}

An immediate consequence of above theorem is
\begin{corollary}\label{sec2,cor2}
Let $f$ and $g$ be permutable transcendental entire functions. Suppose $f$ and $g$ omit two values (need not be same) in the complement of closure of each disk in $\C.$ Then $F(f)=F(g).$
\end{corollary}

\begin{theorem}\label{sec2,thm3}
Let $g$ and $h$ be transcendental entire functions and $f$ be continuous and open such that $f\circ g=h\circ f.$ Suppose there exist a non constant polynomial $P(z)$ and an entire function $K(z)$ such that $P\circ g(z)=K\circ f(z),$ for all $z\in\C.$ If $V$ is a Baker domain of $F(g),$ then $f(V)\subset F(h).$ 
\end{theorem}

\begin{proof}
Let $V\subset F(g)$  be a Baker domain. Then for  $z\in V,$ $g^n(z)\to\ity$  as $n\to\ity.$ Suppose $f(V)\nsubseteqq F(h).$ Then there exist $\al\in V$ such that $f(\al)\in J(h).$ Let $A=max_{|z|=1}|K(z)|.$ As $P(z)$ is a non constant polynomial, there exist $M>0$ such that $|P(z)|>A+1,$ for $|z|>M.$ Also $g^n(z)\to\ity$ for   $z\in V$ as $n\to\ity,$ there exist $m\in\N$ such that $|g^n(z)|>M$ for all $n>m$ and for all $z\in V,$ which implies $|g(z)|>M$ for $z\in g^n(V), n>m.$ As $f(V)$ is an open neighborhood of $f(\al)\in J(h),$ for arbitrarily large $n,$ ($h^n$) takes all values in $f(V)$ with the possible exception of one point. Thus there exist $f(z_0), z_0\in V$ such that for all $n>m, |h^n(f(z_0)|<1.$ Now $1>|h^n(f(z_0))|=|f(g^n(z_0))|$ and so $|K(h^n(f(z_0)))|\leq A.$ As $\be=g^n(z_0)\in g^n(V),$ we have $|g(\be)|>M$ and $|f(\be)|<1.$ Now $|K(h^n(f(z_0)))|=|K(f(g^n(z_0)))|=|P(g(g^n(z_0)))|>A+1,$ which is a contradiction and hence the result.
\end{proof}
We now show that if $f, g$ and $h$ are entire functions satisfying $f\circ g=h\circ f,$ and $f$ and $g$ are of finite order, then derivative of $f\circ g$ has bounded set of asymptotic values.
We first prove a lemma:
\begin{lemma}\label{sec2,lem4}
Let $f$ and $g$ be transcendental entire functions having bounded set of asymptotic values. Then $f.g$ has bounded set of asymptotic values, where $f.g(z)=f(z)g(z),$ for all $z\in\C.$
\end{lemma}

\begin{proof}
Let $h=f.g$ and suppose $AV(h)$ is unbounded. Then there exist a sequence ($z_n$)$\subset AV(h)$ such that $|z_n|>n,$ for all $n\in\N.$ For this sequence  $z_n\in AV(h)$  there exist for each $n\in\N,$ a curve $\Gamma_n\to\ity,$ such that $h(z)\to z_n$ as $z\to\ity$ on $\Gamma_n.$ As $h=f.g,$  therefore $f(z)\to a$ and $g(z)\to b$ as $z\to\ity$ on $\Gamma_n,$ where $a, b\in\ti\C=\C\cup\{\ity\}.$ This implies  $a$ and $b$ are asymptotic values of $f$ and $g$ respectively, and $h(z)\to ab$ as $z\to\ity$ on $\Gamma_n$ and thus $z_n=ab.$  As $AV(f)$ and $AV(g)$ are bounded, so $\{z_n: n\in\N\}$ is  bounded. This is a contradiction and hence the result.
\end{proof}
We will also require the following result:
\begin{theorem}\cite[Theorem 1]{sharma}\label{sec2,thm5}
Let $f, g$ and $h$ be entire functions satisfying $f\circ g=h\circ f.$ If $\rho_f$ and $\rho_g$ are finite, then $\rho_h$ is finite.
\end{theorem}

\begin{theorem}\label{sec2,thm4}
Let $f, g$ and $h$ be entire functions satisfying $f\circ g=h\circ f.$  If $\rho_f$ and $\rho_g$ are finite  and $f'$ and $g'$ have finitely many critical values, then derivative of $f\circ g$ has bounded set of asymptotic values.
\end{theorem}

\begin{proof}
From Theorem \ref{sec2,thm5}, $\rho_h$ is finite. Also for any entire function $F, \rho_F=\rho_{F'},$ we obtain $\rho_f=\rho_{f'}, \rho_g=\rho_{g'}$ and $\rho_h=\rho_{h'}$ are all finite. Also $f'$ and $g'$ have finitely many critical values, therefore $f'$ and $g'$ have finite and hence  bounded set of asymptotic values, \cite[Corollary 3]{berg6}. Consider entire functions $f'(g(z))$ and $g'(z).$ As $AV(f\circ g)\subset AV(f)\cup f(AV(g))$ \cite{berg5}, we have $AV(f'(g(z)))$ is bounded and thus using Lemma \ref{sec2,lem4}, $(f\circ g)'(z)=f'(g(z))g'(z)$ has bounded set of asymptotic values and hence the result.
\end{proof}

Acknowledgement. I am thankful to Prof. A. P. Singh, Central University of Rajasthan for his helpful comments.

\end{document}